\documentclass[11pt]{article}

\usepackage[margin=1in]{geometry}
\usepackage{amsmath,amssymb,amsthm,mathtools}
\usepackage{hyperref}
\usepackage{comment}

\newtheorem{theorem}{Theorem}[section]	
\newtheorem{lemma}{Lemma}[section]
\numberwithin{equation}{section}
\newtheorem{corollary}{Corollary}[section]

\newtheorem{remark}{Remark}[section]
\newtheorem{example}{Example}[section]

\usepackage{upgreek} 
\usepackage{comment}

\def\PP{\mathbb{P}}
\def\RR{\mathbb{R}}
\def\Var{\mathrm{Var}}

\newcommand{\E}{\mathbb E}

\newcommand{\F}{\mathcal F}

\begin{document}
\title{\textbf{A Generalized Elephant Random Walk with Time-Dependent Memory}}

\author{Xiequan Fan
~and~ Yongcheng Qi 
\\ \\
Northeastern University at Qinhuangdao  and  University of Minnesota Duluth 
}

\date{}

\maketitle

\thispagestyle{empty}

\begin{abstract}
The elephant random walk (ERW) is a discrete-time random walk with complete memory. At every step the walker chooses one of the previous increments at random and
then either repeats with a constant probability $p\in [0,1]$ or reverses it.  Since its introduction by Sch\"utz and
Trimper, the model has become a useful probabilistic framework for studying
long-range dependence, anomalous diffusion, reinforcement, and phase
transitions.  This paper proposes a generalized ERW model with time-dependent memory. In this generalized model, 
the usual
constant memory parameter $p$ in the standard ERW is allowed to change with time.  The cental limit theorems, and the law of the iterated logarithm and some other strong laws are obtained.  
\end{abstract}


\newpage

\section{Introduction}\label{intro}

Random walks are among the fundamental objects of probability theory, but the
classical simple random walk is based on the assumption that successive
increments are independent.  Many physical, biological, economic, and
information-processing systems exhibit instead a persistent dependence on
their past.  One natural way to introduce such dependence is to allow the
next increment to depend on a randomly selected increment from the history of
the walk.

The elephant random walk (ERW), introduced by Sch\"utz and Trimper
\cite{SchuetzTrimper2004}, is a particularly simple model of this type. The
name reflects the metaphor that an elephant  can remember where they have been.
The standard one-dimensional ERW is determined by a sequence of increments $X_1,X_2,\ldots$ with 
the position at time $n$ given by $S_n=\sum_{k=1}^n X_k$, $n\ge 1$.
More precisely, the walk starts from the origin, and the first step $X_1$ takes value $+1$ with probability $q$ and $-1$ with probability $1-q$, $q \in [0,1]$.
Given the history
$\F_n=\sigma(X_1,X_2,\ldots,X_n)$, the $\sigma$-algebra generated by $X_1,X_2,\ldots,X_n$, choose 
$U_n$ uniformly from  
$\{ 1, 2,\ldots, n\}$,
and the increment at time $n+1$ is given by
\begin{equation}\label{constant-p}
 X_{n+1}=
 \begin{cases}
 X_{U_n},&\text{with probability } p,\\
 -X_{U_n},&\text{with probability } 1-p,
 \end{cases}
\end{equation}
where the parameter $p\in [0,1]$ measures the strength of memory.  The
conditional mean of $X_{n+1}$ on $\F_n$ is 
\[
 \mathbb E(X_{n+1}\mid\mathcal F_n)
 =\frac{2p-1}{n}S_n.           \]
Thus the ERW is non-Markovian in the position $S_n$, while its asymptotic properties can
be analyzed through a martingale or urn representation.

The principal feature of the classical model is a phase transition at
$p=3/4$. In fact,  $0\leq p<3/4$, $p=3/4$, and $3/4<p\le 1$ are called diffusive regime, critical regime, and super-diffusive regime, respectively.   In the diffusive regime, $S_n$ has fluctuations of
order $\sqrt n$ and obeys a central limit theorem after the usual normalization.  At the critical regime, the normalization acquires a
logarithmic correction.  In the superdiffusive regime, the position $S_n$ is of order $n^{2p-1}$ and converges to a non-Gaussian random variable after suitable normalization.
These three regimes were subsequently established and refined using
martingale, urn, and strong approximation methods; see
\cite{Bercu2018,ColettiGavaSchuetz2017,BaurBertoin2016,HuFeng2022} and numerous properties have been discovered 
\cite{VazquezGuevara2019,KubotaTakei2019,BercuLaulin2021,BercuLaulin2021,Bertoin2022,MaElMachkouriFan2022,GutStadtmuller2023,RoyTakeiTanemuraPower2024,Fang2026}.

The literature has developed in several directions.  One group of papers
alters which part of the past is remembered, including restricted
memory, gradually increasing memory, and smooth amnesia
\cite{GutStadtmuller2021,GutStadtmuller2022,Laulin2022, Aguech2024, RoyTakeiTanemuraPower2024, RoyTakeiTanemura2026}.
Another group expands the distribution of the increments, allowing random or
general step sizes \cite{KissVeto2022,DedeckerFanHuMerlevede2023,FS2024,FLM2026}, or even
polynomially decaying step sizes \cite{Nakano2025}.  Further extensions
include delays or stops \cite{GutStadtmuller2019,GutStadtmullerZeros2021,
GutStadtmullerBernoulli2026}, multidimensional walks
\cite{BercuLaulin2019,Bertenghi2022}, reinforced memory
\cite{Baur2020,Laulin2022Reinforced}, random memory
\cite{DhillonKataria2025} and triangular-array constructions
\cite{RoyTakeiTanemura2025}.  The model has also been connected with random
recursive trees and percolation \cite{Kursten2016}, and statistical methods
for estimating the memory parameter have been developed
\cite{BercuLaulin2024}.   

A separate but relevant line of work predates the modern ERW literature.  Correlated Bernoulli processes with time-varying dependence
parameters were studied in the probability and statistics literature, and
limit theorems for such processes were obtained by James et al.~\cite{JamesJamesQi2008}. However, one of the differences between the ERM model and correlated Bernoulli processes is that all Bernoulli random variables under the framework of correlated Bernoulli processes considered in \cite{JamesJamesQi2008} have the same distribution. A variant of ERW model
in \cite{KubotaTakei2019}
is also  closely related to correlated Bernoulli processes, and 
it can reduce to the classical ERW model with $p\ge 1/2$.

The objective of the present paper is to extend the model \eqref{constant-p} via replacing the constant memory parameer $p$ by $p_n$, where $p_n\in [0,1]$ is an arbitrary sequence of probabilities. 
This model preserves the
uniform choice from the complete past but allows the probability of
repetition versus reversal to vary with time.  It should be distinguished
from models of varying memory in which the memory set itself changes.

The rest of the paper is organized as follows.  In section~\ref{main},  asymptotic properties of generalized ERW, including the central limit theorems, the law of the iterated logarithm, and other almost sure convergence are given, and some refined results under more specific conditions are presented so as to illustrate new results that can be observed for the generalized ERW model. In section~\ref{proof},  the proofs for the main results in the paper are given.

\section{Main results}\label{main}

In this paper, we assume $\{p_n;n\ge 1\}$ is a sequence of real numbers in $[0,1]$. 
Define a sequence of dependent Rademacher
random variables 
$\{X_1,X_2,\ldots\}$ as follows.
First, random $X_1$ has Rademacher distribution $\mathcal{R}(q)$ for some $q\in [0,1]$, that is, $X_1$ 
takes $+1$ with probability $q$ and
$-1$ with probability $1-q$. 
For each $n\ge 1$, conditional on random variables $X_1, \ldots, X_n$,  $\alpha_n$ and $U_n$ are two independent discrete random variables,  
$\alpha_n$ has a Rademacher $\mathcal{R}(p_n)$ distribution
and $U_n$ is uniformly distributed over
$\{1,\ldots,n\}$. Then random variable $X_{n+1}$ is defined as $X_{n+1}=\alpha_nX_{U_n}$. 

For the sequence of random variables $\{X_1,X_2,\ldots\}$ given above, define a random walk whose position at time $n$ is given by
\[
 S_n=\sum_{k=1}^n X_k,~~~n\ge 1.
\]
We call $S_n$ a generalized ERW with time-dependent memory vector $\mathbf{p}:=(p_i; i\ge 1)$. For convenience, we use ERW($\mathbf{p}$) to denote this model.  In particular, if $p_i=p$ is a constant,  
ERW($\mathbf{p}$) is written as ERW($p$), and this is exactly the standard or classical ERW defined in \eqref{constant-p}.

\subsection{Asymptotics for generalized ERW}

We need to introduce the following sequences of constants for normalization of $S_n$.

Set
\begin{equation}\label{an}
   a_n=2p_n-1, ~~~n\ge 1 
\end{equation}
and  define
\begin{equation}\label{An}
A_1= A_2=\frac12, ~~A_n=\frac12\prod_{k=2}^{n-1}\big(1+
 \frac{a_k}{k}\big)^{-1}, ~V_n=\sum^n_{k=2}A_k^2, ~~n\ge 3.
 \end{equation}

Our first theorem is a general result on the law of large numbers for ERW($\mathbf{p}$) when $\lim_{n\to\infty}V_n=\infty$, which will be used in the proofs for the central 
limit theorems and the law of the iterated logarithm.

\begin{theorem}\label{thm1} Assume $V_n\to\infty$ as $n\to\infty$. Then 
\[
\lim_{n\to\infty}\frac{S_n}{n}=0
\]
with probability one.
\end{theorem}

Our next two theorems are the central limit theorem and the law of the iterated logarithm for ERW($\mathbf{p}$) model. 

\begin{theorem}\label{thm2} Assume $V_n\to\infty$ as $n\to\infty$.
Then
\[
\frac{A_nS_n}{\sqrt{V_n}}  \overset{d}\to N(0,1),
\]
where $\overset{d}\to$ denotes convergence in distribution.
\end{theorem}

\begin{theorem}\label{thm3} Assume $V_n\to\infty$ as $n\to\infty$. Then 
\[
\limsup_{n\to\infty}\frac{A_nS_n}{\sqrt{2V_n\ln\ln V_n}}=1 ~~\text{almost surely} 
\]
and
\[
\liminf_{n\to\infty}\frac{A_nS_n}{\sqrt{2V_n\ln\ln V_n}}=-1 ~~\text{almost surely}. 
\]
\end{theorem}

When $\lim_{n\to\infty}V_n<\infty$, ERW($\mathbf{p}$) model behaves drastically differently.

\begin{theorem}\label{thm4} Assume $\lim_{n\to\infty}V_n=\sum^\infty_{k=2}A_k^2<\infty$. Then 
\[
\lim_{n\to\infty}A_nS_n=L ~~\text{almost surely} 
\]
where $L$ is a random variable.  $A_nS_n$ also converges in $\mathcal{L}^4$ to $L$, that is,
\[
\lim_{n\to\infty}\E(A_nS_n-L)^4=0.
\]
Moreover, $L$ is a nondegenerate random variable if and only if $q\in (0,1)$ or at least one of $p_n$ is smaller than $1$.
\end{theorem}

Our next theorem presents some sufficient conditions for non-normality of the limit $L$ in Theorem~\ref{thm4}. 

\begin{theorem}\label{thm5} Assume the limit $L$ in Theorem~\ref{thm4} is non-degenerate.  Then $L$ is  
not normal under each of the following conditions.

\noindent\textbf{(a)}  $q=1/2$ or $p_1=0$;

\noindent\textbf{(b)}  $q\ne \frac12$ and $p_n\ge \frac12$ for all $n\ge 1$.    
\end{theorem}

\vspace{10pt}

\begin{remark}  It is known that when $p_n=p\in (3/4,1]$ for all $n\ge 1$, the limit random variable $L$ in Theorem~\ref{thm4} is not Gaussian; see e.g., \cite{Bercu2018}.  The proof for the non-normality of the limit $L$ for the classical ERW model is based on  its first four moments which have explicit expressions as functions of $p$.  In the present paper,  the moments of $L$ are much more complicated. In part (a) of Theorem~\ref{thm5}, no constrains on $p_n$ is imposed, that is, the limit $L$ is not normal if $\E(X_1)=2q-1=0$. By combining parts (a) and (b) of Theorem~\ref{thm5}, any non-degenerate limit $L$ is not normal if $p_n\ge 1/2$ for all $n$, regardless of the value of $q=\PP(X_1=1)$.  This result reduces to the result for the classical ERW model in the superdiffusive regime.

\end{remark}

\subsection{Phase transition and discussions}

 The classical ERW model has a striking phase transition at $p=3/4$. For $p<3/4$, the walk is diffusive and asymptotically Gaussian on the $\sqrt n$ scale; at $p=3/4$, the scale becomes $\sqrt{n\log n}$; for $p>3/4$,  the walk is superdiffusive and has a non-Gaussian random limit after normalization by $n^{2p-1}$. The generalized ERW has more flexibility and should produce some different results than the classical ERW model.  

Now we discuss the asymptotic properties of the generalized ERW under more specific conditions. We will first consider the situation when $p_n$ converges to some constant probability $p\in [0,1]$, that is,
\begin{equation}\label{limit-p}
\lim_{n\to\infty}p_n=p\in[0,1].
\end{equation}
Our interest is in whether the the asymptotic behavior of ERW($\mathbf{p}$) depends on the limiting memory parameter $p$ only.

For convenience, we first investigate some properties of $A_n$ and $V_n$ under condition \eqref{limit-p}, where $A_n$ and $V_n$ are defined in \eqref{An}.  

Define 
\[
\omega(x):=x-\ln(1+x), ~~x>-1.
\]
Then
one can easily check that
\[
-x^2\le \omega(x)\le x^2,  ~~~x\in [-\frac12, \frac32].
\]
Also note that for each $k\ge 2$, $-1/2\le \frac{2p_k-1}{k}\le 1/2$ and $-1/2\le \frac{2p-1}{k}\le 1/2$. We have
\begin{equation}\label{convergent}
\sum^\infty_{k=2}|\omega(\frac{2p_k-1}{k})|\le \sum^\infty_{k=2}\frac{1}{k^2}<\infty ~\text{ and }~
\sum^\infty_{k=2}|\omega(\frac{2p-1}{k})|\le \sum^\infty_{k=2}\frac{1}{k^2}<\infty.
\end{equation}
By combining the following two equations  
\[
-\sum^{n-1}_{k=2}\ln(1+\frac{2p_k-1}{k})=-\sum^{n-1}_{k=2}\frac{2p_k-1}{k}+\sum^{n-1}_{k=2}\omega(\frac{2p_k-1}{k}),~n\ge3
\]
and 
\[
-\sum^{n-1}_{k=2}\ln(1+\frac{2p-1}{k})=-\sum^{n-1}_{k=2}\frac{2p-1}{k}+\sum^{n-1}_{k=2}\omega(\frac{2p-1}{k}),~n\ge 3
\]
we have
\begin{equation}\label{ratio}
\frac{\prod^{n-1}_{k=2}\big(1+\frac{2p_k-1}{k}\big)^{-1}}{\prod^{n-1}_{k=2}\big(1+\frac{2p-1}{k}\big)^{-1}}=\exp\left\{-2\sum^{n-1}_{k=2}\frac{p_k-p}{k}+\sum^{n-1}_{k=2}\omega(\frac{2p_k-1}{k})-\sum^{n-1}_{k=2}\omega(\frac{2p-1}{k})\right\}.
\end{equation}

Denote $\Gamma$ as the gamma function, that is,
\[
\Gamma(x)=\int^\infty_0t^{x-1}e^{-t}dt,~~x>0.
\]
Then we obtain
\[
\prod^{n-1}_{k=2}\big(1+\frac{2p-1}{k}\big)^{-1}=\frac{\prod^{n-1}_{k=2}k}{\prod^{n-1}_{k=2}(k+2p-1)}=\frac{\Gamma(n)\Gamma(1+2p)}{\Gamma(n+2p-1)}.
\]
which, coupled with Stirling’s formula, yields $\Gamma(n)/\Gamma(n+2p-1)=n^{1-2p}\big(1+o(1)\big)$, and thus
\begin{equation}\label{app}
    \prod^{n-1}_{k=2}\big(1+\frac{2p-1}{k}\big)^{-1}=\Gamma(1+2p)n^{1-2p}\big(1+o(1)\big)
\end{equation}
as $n\to\infty$.  Now combining \eqref{convergent}, \eqref{ratio} and \eqref{app} we can write
\begin{equation}\label{An-app}
 A_n=C_nH_nn^{1-2p},   
\end{equation}
where $\lim_{n\to\infty}C_n=C$ for some constant $C\in  (0,\infty)$, and 
\begin{equation}\label{Hn}    
H_n=\exp\left\{-2\sum^{n-1}_{k=2}\frac{p_k-p}{k}\right\}.
\end{equation}

Recall that a function $\ell(x)$, defined on $(c,\infty)$ for some $c>0$, is slowly varying is 
\[
\lim_{t\to\infty}\frac{\ell(xt)}{\ell(t)}=1~~\text{for any }x>0.
\]
A function $\varphi$ is regularly varying with index $\alpha$ if $\varphi(x)=x^{\alpha}\ell(x)$ for some slowly varying function $\ell(x)$. 

One can easily verify that $H_n$ is slowly varying under condition \eqref{limit-p} in the sense that 
$\ell(x)=H_{\lceil x \rceil}$ is slowly varying, where
$\lceil x \rceil$ denotes
the smallest integer larger than or equal to $x$.  

Of course, $C_nH_n$ is also slowly varying. From \eqref{An-app}, $A_n$ is regularly varying with index $\alpha=1-2p$, and $A_n^2$ is regularly varying with index $2(1-2p)$. 

From the properties of regularly varying functions,
\[
\lim_{x\to\infty}\frac{\ell(x)}{\ln x}=0 ~~\text{if $\ell(x)$ is slowly varying.} 
\]
We have
$\sum^\infty_{k=2}A_k^2=\infty$ if $p<\frac{3}{4}$, or equivalently $2(1-2p)>-1$, and $\sum^\infty_{k=2}A_k^2<\infty$ if $p>\frac{3}{4}$.

Furthermore, from Karamata’s Theorem (see,
e.g., Theorem~1.5.11 in \cite{BGT1987}), we have for $p<3/4$
\begin{equation}\label{Vnestimation}
V_n=\sum^n_{k=2}A_k^2=\int^n_1A^2_{\lceil x
\rceil}dx=\frac{nA_n^2}{2(1-2p)+1}\big(1+o(1)\big)=
\frac{nA_n^2}{3-4p}\big(1+o(1)\big)
\end{equation}
as $n\to\infty$.  When \eqref{limit-p} holds with $p<3/4$, we have from the above discussion that
\[
\frac{S_n}{\sqrt{n}}=\frac{1+o(1)}{\sqrt{3-4p}}\frac{A_n}{\sqrt{V_n}}~\text{ and }~
\frac{S_n}{\sqrt{2n\ln\ln n}}=\frac{1+o(1)}{\sqrt{3-4p}}\frac{A_nS_n}{\sqrt{2V_n\ln\ln V_n}}.
\]
Immediately from Theorems~\ref{thm2} and \ref{thm3} we have 

\begin{corollary}\label{cor1} Assume $\lim_{n\to\infty}p_n=p\in [0,3/4)$.  Then we have
\[
\frac{S_n}{\sqrt{n}}\overset{d}\to N(0,\frac{1}{3-4p})
\]
and
\[
\limsup_{n\to\infty}\frac{\pm S_n}{\sqrt{2n\ln\ln n}}=\frac1{\sqrt{3-4p}}
 ~~\text{almost surely}.
 \]
\end{corollary}

\vspace{20pt}

Clearly,  the above central limit theorem and the law of the iterated  logarithm are exactly the same as in those in \cite{Bercu2018} for the case when $p_n=p$.  It means that the convergence rate of $p_n$ to $p$
for $p<3/4$ does not change the asymptotic behaviors of $S_n$.  However, our further investigation shows that it is no longer the case when $p=3/4$ or $p>3/4$. In fact,  for both cases,  the slowly varying function $H_n$ defined in \eqref{Hn} greatly impact the asymptotic behaviors of $S_n$.

When $p\in (3/4, 1]$, we have $\sum^\infty_{k=2}A_k^2<\infty$ as aforementioned.. Using the representation of $A_n$ given in \eqref{An-app} and noting that $C_n\to C\in (0,\infty)$, we conclude the following corollary from
Theorem~\ref{thm4}.

\begin{corollary}\label{cor2} Assume $\lim_{n\to\infty}p_n=p\in (3/4, 1]$.  Then we have
\[
\lim_{n\to\infty}\frac{H_nS_n}{n^{2p-1}}=L  ~~\text{almost surely}
\]
almost surely, where $L$ is a non-degenerate random variable if $q\in (0,1)$ or at least one $p_n$ is not equal to one. 
\end{corollary}

\vspace{20pt}

\begin{example}\label{exam1} Assume $3/4<p\le 1$.  Let $a\in \RR$ be a fixed number and $n_0\ge 2$ be integer such that  $-p\le a/\ln(n_0)\le 1-p$. Define 
$p_n=p+na\big(\ln\ln(n+1)-\ln\ln(n)\big)$ for $n\ge n_0+1$, and 
$p_1, \ldots, p_{n_0}$ are any $n_0$ given numbers in $[0,1]$.  Note that 
\[
0<n\big(\ln\ln(n+1)-\ln\ln(n)\big)=\int^{1}_0\frac{n}{(n+t)\ln(n+t)}dt<\frac{1}{\ln(n)}\to 0
\]
as $n\to\infty$. Then it is easy to see $0\le p_n\le 1$ for all $n>n_0$. 
From the definition we have
\[
H_n=\exp\left\{-2\sum^{n-1}_{k=2}\frac{p_k-p}{k}\right\}=
\exp\left\{-2a\ln\ln(n)-2\sum^{n_0}_{k=2}\frac{p_k-p}{k}-2a\ln\ln(n_0)\right\}=(\ln n)^{-2a}d
\]
for all $n>n_0$, where $d=\exp\left\{-2\sum^{n_0}_{k=2}\frac{p_k-p}{k}-2a\ln\ln(n_0)\right\}$.
It follows from Corollary~\ref{cor2} that
\[
\lim_{n\to\infty}\frac{S_n}{n^{2p-1}\ln^{2a} (n)}=L_1
\]
almost surely, where $L_1$ is a non-degenerate random variable.   We remark that  $a\le 0$ is required when $p=1$. 
\end{example}

\begin{example}\label{exam2} Assume $p\in (-3/4,1]$ and $a\in \RR$ is a fixed number.  When $p=1$,  we also assume $a\ge 0$. Assume $\beta\in (0,1)$.  
Define $p_n=p+an\big(\ln^{\beta}(n+1)-\ln^{\beta}(n)\big)$ for all large $n$. One can show that $p_n\in [0,1]$ for all $n>n_0$ for some integer $n_0$. As in Example~\ref{exam1},  $p_1, \ldots, p_{n_0}$ are any given numbers in $[0,1]$. Then we have that
 $H_n$ is proportional to $\exp\{-2a\ln^{\beta}(n)\}$ for all large $n$.  Then it follows from Corollary~\ref{cor2} that
\[
\lim_{n\to\infty}\frac{S_n}{n^{2p-1}\exp\{2a\ln^{\beta}(n)\}}=L_2
\]
almost surely, where $L_2$ is a non-degenerate random variable.
\end{example}

The case when $p=3/4$ is of particular interest but condition \eqref{limit-p} alone does not imply  $\lim_{n\to\infty}V_n=\infty$  or $\lim_{n\to\infty}V_n<\infty$.
From our previous discussion,  $A_n^2$ is regularly varying with index $-1$ when $p=3/4$.  In this case, there is no general estimation for $V_n$ like in \eqref{Vnestimation}.
Therefore, a general discussion seems difficulty. In what follows,  we will adapt Examples~\ref{exam1} and \ref{exam2} to show some quite different behaviors in the case $p=3/4$.

\begin{example}\label{exam3} Assume $p=3/4$ and define $p_n$ in the same way as in Example~\ref{exam1}.  That is, 
we define $p_n=3/4+na\big(\ln\ln(n+1)-\ln\ln(n)\big)$ for $n\ge n_0+1$, where $n_0$ satisfies the same condition  as in Example~\ref{exam1}.  The computation for $H_n$ remains valid, that is,   $H_n=(\ln n)^{-2a}d$ for all $n\ge n_0+1$.
According to \eqref{An-app}, we have 
\[
V_n=\sum^n_{n=2}A_k^2=\sum^n_{k=2}C_k^2H_k^2k^{-1}
\begin{cases}
\sim \frac{C^2d^2}{1-4a}\ln^{1-4a}(n) &\text{ if }~a<1/4;\\
\sim C^2d^2\ln\ln n   & \text{ if }     ~ a=1/4;\\
\text{is bounded}     & \text{ if }     ~  a>1/4.
\end{cases}
\]
Therefore, $V_n$ diverges when $a\le 1/4$, 
\[
\frac{\sqrt{V_n}}{A_n}\sim
\begin{cases}
 \frac{1}{\sqrt{1-4a}}\sqrt{n\ln^{1-4a}(n)} &\text{ if }~a<1/4;\\
 \sqrt{n\ln_2(n)}   & \text{ if }     ~ a=1/4,
\end{cases}
\]
and
\[
\frac{\sqrt{2V_n\ln\ln V_n}}{A_n}\sim
\begin{cases}
 \frac{1}{\sqrt{1-4a}}\sqrt{2n\ln^{1-4a}(n)\ln_3(n)} &\text{ if }~a<1/4;\\
 \sqrt{2n\ln_2(n)\ln_4(n)}   & \text{ if }     ~ a=1/4,
\end{cases}
\]
where $\ln_{j+1}(n)=\ln(\ln_{j}(n))$ for $j\ge 1$ with $\ln_1(n)=\ln n$.
By using Theorems~\ref{thm2} and \ref{thm3} we have
\[
\frac{S_n}{\sqrt{n\ln^{1-4a}(n)}}\overset{d}\to N(0, 1-4a)~~~\text{ if } a>1/4; 
\]
\[
\frac{S_n}{\sqrt{n\ln_2n}}\overset{d}\to N(0, 1)~~~\text{ if } a=1/4, 
\]
\[
\limsup_{n\to\infty}\frac{\pm S_n}{\sqrt{2n\ln^{1-4a}(n)\ln_3(n)}}=\sqrt{1-4a}~~~\text{almost surely if }a<1/4
\]
and 
\[
\limsup_{n\to\infty}\frac{\pm S_n}{\sqrt{2n\ln_2(n)\ln_4(n)}}=1
~~~\text{almost surely if }a=1/4.
\]

When $a>1/4$, we apply Theorem~\ref{thm4} and obtain  
\[
\lim_{n\to\infty}\frac{S_n}{\sqrt{n\ln^{2a}(n)}}=L_3
\]
for some non-degenerate random variable $L_3$.
\end{example}

\begin{example}\label{exam4} Assume $p=3/4$,  $a\in \RR$ and $\beta\in (0,1)$.  Define $p_n$
as in Example~\ref{exam2}, that is, $p_n=3/4+an\big(\ln^{\beta}(n+1)-\ln^{\beta}(n)\big)$ for all large $n$. Again, $H_n$ is proportional to $\exp\{-2a\ln^{\beta}(n)\}$ for all large $n$,  that is,  $H_n=d\exp\{-2a\ln^{\beta}(n)\}$ for large $n$, where $d>0$ is a constant.

It is easy to check that $\lim_{n\to\infty}V_n=\infty$ if $a\le  0$ and 
$\lim_{n\to\infty}V_n<\infty$ if $a>0$.  Moreover, we have
for $a\le 0$
\[
V_n=\sum^n_{k=2}C_n^2H_k^2k^{-1}\sim C^2d^2\int^{n}_1\frac{\exp\{-4a\ln^{\beta}(t)\}}{t}dt=
C^2d^2\int^{\ln n}_0\exp\{-4ay^\beta\}dy.
\]
Since 
\[
\int^{\ln n}_0\exp\{-4ay^\beta\}dy
\begin{cases}
=\ln n & \text {if }~a=0;\\
\sim \frac{1}{-4a\beta}\ln^{1-\beta}(n)\exp\{-4a\ln^{\beta}(n)\} &  \text {if }~a<0,\
\end{cases}
\]
we have
\[
V_n\sim
\begin{cases}
C^2d^2\ln n & \text { if }~a=0;\\
\frac{C^2d^2}{-4a\beta}\ln^{1-\beta}(n)\exp\{-4a\ln^{\beta}(n)\} &  \text { if }~a<0.\
\end{cases}
\]
From Theorems~\ref{thm2} and \ref{thm3} we conclude 
if $a=0$
\begin{equation}\label{CLT-34}
\frac{S_n}{\sqrt{n\ln n}}\overset{d}\to N(0, 1),
\end{equation}
\begin{equation}\label{LIL-34}
    \limsup_{n\to\infty}\frac{\pm S_n}{\sqrt{2n\ln n\ln_3(n)}}=1~~~\text{ almost surely},
\end{equation}
and  if $a<0$
\[
\frac{S_n}{\sqrt{n\ln^{1-\beta}(n)}}\overset{d}\to N(0,\frac{1}{-4a\beta}),
\]
\[
\limsup_{n\to\infty}\frac{\pm S_n}{\sqrt{n\ln^{1-\beta}(n)\ln\ln n}}=\frac{1}{\sqrt{-2a}}~~~\text{ almost surely.}
\]
\end{example}

\begin{remark}\label{remark2}  In both Examples~\ref{exam3} and \ref{exam4}, the results for special case $a=0$, i.e,  \eqref{CLT-34} and \eqref{LIL-34}, correspond to the central limit theorem  and the law of the iterated logarithm for
the classical ERW in the critical regime; see, e.g., Theorems 3.6 and 3.5 in \cite{Bercu2018}.  
Our examples here reveal many new types of convergence rates from our generalized ERW model. Since the almost sure convergence in Theorem~\ref{thm4} also implies the converge in distribution with rate $1/A_n$, which is of the same order as $\sqrt{V_n}/A_n$. Combining Theorem~\ref{thm2} and \ref{thm4},  $\sqrt{V_n}/A_n$ 
seems a unified expression for the rate of distributional convergence.  From part (c) of Lemma~\ref{lem1}, $\sqrt{V_n}/A_n$ is between $\sqrt{n/3}$ and $n$.

\end{remark}

\section{Proofs}\label{proof}

For each $n\ge 1$, write  
$\mathcal F_n=\sigma(X_1,\ldots,X_n)$ as in section~\ref{intro}. 
Define $\mathcal F_0=\{\emptyset, \Omega\}$ as the trivial $\sigma$-algebra. 

It is easy to see that
\begin{equation}\label{conditional-mean}
 \mathbb E(X_n\mid\mathcal F_{n-1})
 =\frac{a_{n-1}}{n-1}S_{n-1}, ~n\ge 2   
\end{equation}
and
\begin{equation}\label{eq:mean}
 \mathbb E(X_1\mid\mathcal F_0)=\mathbb E(X_1)=2q-1.    \end{equation}

Since $-1\le a_n\le 1$,  
we have $1+\frac{a_n}{n}>0$ for all $n\ge 2$, and $A_n\in (0,\infty)$.
Note that $1+a_1=2p_1\ge 0$,  and $1+a_1=0$ only if $p_1=0$. 

We now define a normalized ERW by
\[
M_0=p_1(2q-1),~~M_1=\frac12(1+a_1)X_1=p_1X_1,~~M_n=A_nS_n,~~n\ge 2.
\]
By using \eqref{conditional-mean}, \eqref{eq:mean} and 
\eqref{An},  it is easy to check that
\[
\E(M_n|\mathcal F_{n-1})=M_{n-1},~~~n\ge 1; 
\]
that is,  $\{M_n, \mathcal F_n; n\ge 0\}$ is a martingale sequence. We now define a sequence of martingale
differences $\{\delta_n, \mathcal F_n; n\ge 1\}$ 
by
\[
\delta_n=M_n-M_{n-1}, ~~n\ge 1
\]
since $\E(\delta_n|\mathcal F_{n-1})=0$ for $n\ge 1$.  More specifically, we have
\begin{equation}\label{deltan}
\delta_1=p_1\Big(X_1-(2q-1)\Big), ~~\delta_n=A_{n}\left(X_{n}-a_{n-1}\frac{S_{n-1}}{n-1}\right),~n\ge 2,
\end{equation}
\begin{equation}\label{Conditional-2}
\E(\delta_1^2|\F_0)=4p_1^2q(1-q),~~\E[\delta_n^2\mid\F_{n-1}]=A_{n}^2\left(1-a_{n-1}^2(\frac{S_{n-1}}{n-1})^2\right), ~n\ge 2,
\end{equation}
and
\begin{equation}\label{AnSn}
    A_nS_n=M_n=\sum^n_{k=2}\delta_k+\delta_1+2p_1(2q-1)=
\sum^n_{k=2}\delta_k+p_1X_1,~~n\ge 2.
\end{equation}
Since $|a_n|\le 1$, $|S_n/n|\le 1$ and $|X_n|=1$ for all $n\ge 1$, we have the definition of $\delta_n$ and equation \eqref{Conditional-2} that
\begin{equation}\label{bound}
|\delta_n|\le 2A_n, ~~~\E(\delta_n^2|\F_{n-1})\le A_n^2,~~~n\ge 2.
\end{equation}

We remark that a useful property of $A_n$ defined in \eqref{An} is
\begin{equation}\label{An+1}
A_1=A_2=\frac12, ~~A_{n+1}=A_n\big(1+\frac{a_n}{n}\big)^{-1},~~n\ge 2.
\end{equation}

We consider the asymptotic properties of $\sum^n_{k=2}\delta_n$ only. 
The properties of $A_nS_n$ can be obtained via \eqref{AnSn}
since $|\delta_1|$ is bounded by $1$.

The following lemma  will be used in our proofs.

\begin{lemma}\label{lem1} For the sequences $a_n$, $A_n$ and $V_n$ defined in \eqref{an} and \eqref{An} we have

\noindent\textbf{(a)}:  $nA_n$ is non-decreasing in $n\ge 2$.

\noindent\textbf{(b)}:  $A_n/(n-1)$ is non-increasing in $n\ge 2$.

\noindent\textbf{(c)}:  $\frac{1}{n}\le \frac{A_n}{\sqrt{V_n}}<\sqrt{\frac{3}{n}}$ for $n\ge 2$.
\end{lemma}

\begin{proof}  For convenience, we interpret $\prod^n_{k=2}x_k=1$ for any sequence $\{x_k\}$ whenever $n<2$.  Then $A_n=\frac12\prod^{n-1}_{k=2}\big(1+\frac{a_k}k\big)^{-1}$ and 
$n=2\prod^{n-1}_{k=2}(1+\frac1k)$ for $n\ge 2$.
For $n\ge 2$,
\[
\frac{(n+1)A_{n+1}}{nA_n}=\frac{1+\frac{1}{n}}{1+\frac{a_n}{n}}\ge 1 
\]
since $a_n=2p_n-1\le 1$ for all $n\ge 2$.  This proves part (a) of the lemma.

To show part (b),  we see that $1/(n-1)=\prod^{n-1}_{k=1}(1-\frac{1}{k})$ for $n\ge 2$, and
\[
\frac{A_{n+1}/n}{A_n/(n-1)}=\frac{1-\frac{1}{n}}{1+\frac{1}{a_n}}
=\frac{1-\frac{1}{n}}{1+\frac{2p_1-1}{n}}\le 1
\]
since $a_n\ge -1$.  Thus, we  conclude that $A_n/(n-1)$ is non-increasing.  

By using the monotonicity, we have from part (a)
\[
V_n=\sum^n_{k=2}A_k^2=\sum^n_{k=2}\frac{(kA_k)^2}{k^2}\le (nA_n)^2\sum^n_{k=2}\frac{1}{k^2}<A_n^2n^2
\]
and  from part (n)
\[
V_n=\sum^n_{k=2}(\frac{A_k}{k-1})^2(k-1)^2\ge (\frac{A_n}{n-1})^2\sum^n_{k=2}(k-1)^2=\frac{A_n^2}{(n-1)^2}\frac{1}{6}n(n-1)(2n-1)>\frac{n}3.
\]
Two estimates above yield part (c).
This completes the proof.
\end{proof}

\begin{lemma}\label{lem2} If $\lim_{n\to\infty}V_n=\infty$, then $nA_n$ is monotone and diverges to infinity.
\end{lemma}

\begin{proof} The monotonicity of $nA_n$ follows from part (a) of Lemma~\ref{lem1}.  
Set $b=\lim_{n\to\infty}nA_n$. Then $b=\infty$.  Otherwise,  $b$ is finite and $A_n\le b/n$ for all $n\ge 2$ and
$V_n\le \sum^n_{k=2}b^2/k^2<b^2$ for all $n\ge 2$, contradictory to the assumption.  The proof is completed.
\end{proof}

\begin{lemma}\label{lem3}
Assume $\{b_n; n\ge 2\}$ is sequence of positive constants.  If $\sum^\infty_{n= 2}b_n^{-2}A_n^2<\infty$, then $\sum^n_{k=2}b_k^{-1}\delta_k$ converges almost surely.
\end{lemma}

\begin{proof} Note that $\{b_n^{-1}\delta_n, \F_n; n\ge 2\}$ is a sequence of martingale differences.  It follows from  Theorem~2.17 in Hall and Heyde~\cite{Hall1980} that  $\sum^n_{k=2}b_k^{-1}\delta_k$ converges almost surely if $\sum^\infty_{n=2}b_n^{-2}\E(\delta_n^2|\F_{n-1})<\infty$ almost surely. Then we have
\[
\sum^\infty_{n=2}b_n^{-2}\E(\delta_n^2|\F_{n-1})\le \sum^\infty_{n=2}b_n^{-2}A_n^2 <\infty
\]
from the estimates in \eqref{bound} and our assumption in the theorem.    This completes the proof.
\end{proof}

Now we are ready to prove the theorems in section~\ref{main}.

\begin{proof}[Proof of Theorem~\ref{thm1}] 
We will apply Lemma~\ref{lem3} with $b_n=nA_n$. With this choice,  we have
\[
\sum^\infty_{n=2}b_n^{-2}A_n^2=\sum^\infty_{n=2}\frac{1}{n^2}<\infty,
\]
which implies $\sum^n_{k=2}\frac{\delta_k}{kA_k}$ converges almost surely. Moreover, since $b_n=nA_n$ is non-decreasing and diverges to infinity, we have from Kronecker's lemma that 
$\frac{1}{nA_n}\sum^n_{k=2}\delta_k$ converges almost surely.  In view of \eqref{AnSn}, we obtain
\[
\frac{S_n}{n}=\frac{A_nS_n}{nA_n}=\frac{\sum^n_{k=2}\delta_k}{nA_n}+\frac{p_1X_1}{nA_n}
\]
converges almost surely to zero.   This completes the proof.
\end{proof}

\begin{proof}[Proof of Theorem~\ref{thm2}] According to Corollary 3.1 in \cite{Hall1980}, it suffices to verify
\begin{equation}\label{inprobability}
    \frac1{V_n}\sum^n_{k=2}\E(\delta_k^2|\F_{k-1})\to 1~~~\text{in probability} 
\end{equation}
and the conditional Lindeberg condition
\begin{equation}\label{lindeberg}   \frac1{V_n}\sum^n_{k=2}\E(\delta_k^2I(|\delta_k|\ge \varepsilon\sqrt{V_n})|\F_{k-1})\to 0~~~\text{in probability} 
\end{equation}
for every $\varepsilon>0$.

Both conditions can be easily verified. For \eqref{inprobability}, we have from Theorem~\ref{thm1} and \eqref{Conditional-2} that $\E(\delta_k^2|\F_{k-1})=A_k^2(1+o(1))$ almost surely as $k\to\infty$, which implies  
\[
\frac{1}{V_n}\sum^n_{k=2}\E(\delta_k^2|\F_{k-1})=
\frac1{V_n}\sum^n_{k=2}A_k^2(1+o(1))\to 1
\]
almost surely. This is stronger than \eqref{inprobability}.

For every  fixed integer $k_0\ge 2$, we have from part (c) in Lemma~\ref{lem1} that
\[
\frac{\max_{2\le k\le n}A_k}{\sqrt{V_n}}\le \frac{\max_{2\le k\le k_0}A_k}{\sqrt{V_n}}+\frac{\max_{k_0+1\le k\le n}(A_k/\sqrt{V_k})\sqrt{V_k}}{\sqrt{V_n}}\le \frac{\max_{2\le k\le k_0}A_k}{\sqrt{V_n}}+\sqrt{\frac{3}{k_0+1}},
\]
which, coupled with assumption $\lim_{n\to\infty}V_n=\infty$, yields 
\[
\limsup_{n\to\infty}\frac{\max_{2\le k\le n}A_k}{\sqrt{V_n}}\le \sqrt{\frac{3}{k_0+1}}.
\]
Since $k_0$ is arbitrary, we have by letting $k_0\to\infty$ 
\begin{equation}\label{negligible}
    \lim_{n\to\infty}\frac{\max_{2\le k\le n}A_k}{\sqrt{V_n}}=0.
\end{equation}
Since $|\delta_k|\le 2A_k$ for $k\ge 2$, for any fixed $\varepsilon>0$, if $n$ is large enough such that $2\max_{1\le k\le n}A_k/\sqrt{V_n}<\varepsilon$, the sets $\{|\delta_k|\ge \varepsilon\sqrt{V_n}\}$, $2\le k\le n$, are empty, and the left-hand side of \eqref{lindeberg} are zero almost surely for all large $n$. This proves \eqref{lindeberg} and thus completes the proof of Theorem~\ref{thm2}.
\end{proof}

\begin{proof}[Proof of Theorem~\ref{thm3}] 
To show the law of the iterated logarithm in the theorem, we will employ 
Theorems 4.7 and 4.8 in \cite{Hall1980}.  From these two theorems, it suffices to check the following three conditions:
\begin{equation}\label{almostsure}
    \frac1{V_n}\sum^n_{k=2}\E(\delta_k^2|\F_{k-1})\to 1~~~\text{almost surely}; 
\end{equation}
\begin{equation}\label{series}
    \sum^\infty_{k=2}\frac{1}{V_k^2}\E(\delta_k^4|\F_{k-1})<\infty~~~\text{almost surely}; 
\end{equation}
\begin{equation}\label{limit1}
    \lim_{n\to\infty}\frac{V_n}{V_{n+1}}=1 \text{ and } 
    \lim_{n\to\infty}V_n=\infty. 
\end{equation}
In fact,  \eqref{almostsure} has been verified in the proof of 
Theorem~\ref{thm2}. The first limit in \eqref{limit1} follows from 
\eqref{negligible} and the second one is our assumption in the theorem. Therefore, we need to show \eqref{series} only. 

Since $|\delta_n|\le 2A_n$,  we obtain from part (c) in Lemma~\ref{lem2} that
\[
\sum^\infty_{k=2}\frac{1}{V_k^2}\E(\delta_k^4|\F_{k-1})\le 16\sum^\infty_{k=2}\frac{A_k^4}{V_k^2}\le 16\sum^\infty_{k=2}(\frac{3}{k})^2=144\sum^\infty_{k=2}\frac1{k^2}<\infty,
\]
which proves \eqref{series}. Hence, we have proved the theorem.
  \end{proof}

\begin{proof}[Proof of Theorem~\ref{thm4}]
From \eqref{bound}, we have
\[
\sum^\infty_{k=2}\E(\delta_k^2|\F_{k-1})\le \sum^\infty_{k=2}A_k^2<\infty
\]
from our assumption.  
    By setting $b_n=1$ for all $n$ in the proof of Theorem~\ref{thm1}, we have $\sum^n_{k=2}\delta_k$
    converges almost surely to a random variable  $\sum^\infty_{k=2}\delta_k$.
    
    Therefore, we have from \eqref{AnSn}
 that
 \[
\lim_{n\to\infty}A_nS_n=
\lim_{n\to\infty}\sum^n_{k=2}\delta_k
+\delta_1+2p_1(2q-1)
=\sum^\infty_{k=1}\delta_k+2p_1(2q-1)=:L
 \]
 almost surely.
 From Burkholder's inequality for martingales, see, e.g., Theorem 2.10 in \cite{Hall1980}, we have
\[
\E(A_nS_n-L)^4=\E(\sum^\infty_{k=n+1}\delta_k)^4\le (36\sqrt{3})^4\Big(\sum^\infty_{k=n+1}\E(\delta_k^2)\Big)^2\le
16(36\sqrt{3})^4\Big(\sum^\infty_{k=n+1}A_k^2\Big)^2
\]
 converges to zero as $n\to\infty$. We have used the estimate $|\delta_k|\le 2A_k$ in \eqref{bound} and the assumption $\sum^\infty_{k=2}A_k^2<\infty$.

To show the necessary and sufficient condition for the nondegeneracy of $L$, i.e.,  $\Var(L)>0$,
it is equivalent to showing that $\Var(L)=0$ if and only if $\Var(X_1)=q(1-q)=0$ and $p_n=1$ for all $n\ge 1$. 

 First, note that if $p_n=1$ for all $n\ge 1$, then $X_n=X_1$ for all $n\ge 1$, and $A_nS_n=S_n/n=X_1$ for all $n\ge 2$. This implies that the limit $L=X_1$ is degenerate if $\Var(X_1)=0$  and $p_n=1$ for all $n\ge 1$, and thus we have proved the sufficiency.  

To show the necessity, it suffices to show that $\Var(L)>0$ under one of the following two conditions:

\noindent (\textbf{C1}): For some $n_0\ge 1$, $p_{n_0}<1$; 

\noindent (\textbf{C2}): $p_n=1$ for all $n\ge 1$ and $q \in (0, 1)$.   

We have shown under condition $p_n=1$ for all $n\ge 1$ that  
the limit $L=X_1$.  Therefore, $L$ is nondegenerate 
under condition (\textbf{C2}).    To show $\Var(L)>0$ under condition
(\textbf{C1}), we note that
\begin{equation}\label{Var(L)}
\Var(L)=\sum^\infty_{n=1}\E(\delta_n^2)=
4p_1^2q(1-q)+\sum^\infty_{n=2}A_{n}^2\Big(1-a_{n-1}^2\E(\frac{S_{n-1}}{n-1})^2\Big)
\end{equation}
from \eqref{Conditional-2}.  Assume $p_{n_0}<1$ for some $n_0\ge 1$, and  thus,  $\alpha_{n_0}=-1$ with a positive probability $1-p_{n_0}>0$, and $X_{n_0+1}=-X_{\beta_{n_0}}$.  $|S_{n_0+1}|\le n_0-1$ with a probability $1-p_{n_0}$. Therefore, we have
\[
A_{n_0+2}^2\Big(1-a_{n_0+1}^2\E(\frac{S_{n_0+1}}{n_0+1}\big)^2\Big)\ge A_{n_0+1}^2\Big(1-(\frac{n_0-1}{n_0+1})^2\Big)(1-p_{n_0})>0
\]
which implies $\Var(L)>0$ from \eqref{Var(L)}.
 This completes the proof of the theorem.
 \end{proof}

\begin{proof}[Proof of Theorem~\ref{thm5}(\textbf{a})]
Set
\begin{equation}\label{EnQn}
E_n=\E(S_n^2), ~~Q_n=\E(S_n^4),
\end{equation}
and define
\[
K_n=E(S_n^4)-3\big(E(S_n^2)\big)^2
=Q_n-3E_n^2.
\]
We first derive a recursive formula for $K_n$ without any assumption on $X_1$ or $p_n$.

Since
\[
S_{n+1}=S_n+X_{n+1}, ~~X_{n+1}^2=1,
\]
we have
\[
S_{n+1}^2=S_n^2+2S_nX_{n+1}+1.
\]
By taking conditional expectations, we have from \eqref{conditional-mean}
\[
\E(S_{n+1}^2\mid\mathcal F_n)
=
\left(1+\frac{2a_n}{n}\right)S_n^2+1,
\]
and consequently,
\begin{equation}\label{En+1}
E_{n+1}
=
\left(1+\frac{2a_n}{n}\right)E_n+1.
\end{equation}

Similarly, since
\[
S_{n+1}^4
=
S_n^4+4S_n^3X_{n+1}
+6S_n^2+4S_nX_{n+1}+1,
\]
we obtain
\[
E(S_{n+1}^4\mid\mathcal F_n)
=
\left(1+\frac{4a_n}{n}\right)S_n^4
+
\left(6+\frac{4a_n}{n}\right)S_n^2+1.
\]
Then
\begin{equation}\label{Qn+1}
Q_{n+1}
=
\left(1+\frac{4a_n}{n}\right)Q_n
+
\left(6+\frac{4a_n}{n}\right)E_n+1.
\end{equation}

Combining \eqref{En+1}, \eqref{Qn+1} and definition \eqref{EnQn}, we
have 
\begin{eqnarray*}
K_{n+1}
&=&
\left(1+\frac{4a_n}{n}\right)Q_n
+\left(6+\frac{4a_n}{n}\right)E_n+1
-3\left[
\left(1+\frac{2a_n}{n}\right)E_n+1
\right]^2\\
&=&
\left(1+\frac{4a_n}{n}\right)(Q_n-3E_n^2)
+\left(6+\frac{4a_n}{n}\right)E_n+1\\
&&\quad
-3\left[
\frac{4a_n^2}{n^2}E_n^2
+2\left(1+\frac{2a_n}{n}\right)E_n+1
\right],
\end{eqnarray*}
which simplifies to 
the fundamental recursion for all $n\ge 1$
\begin{equation}\label{Kn+1}
K_{n+1}
=\left(1+\frac{4a_n}{n}\right)K_n-\frac{12a_n^2}{n^2}E_n^2
-\frac{8a_n}{n}E_n-2.
\end{equation}

We now split our proof in a few steeps,

\noindent \textbf{Step 1}. We will show $K_4\le 0$ and $K_5\le -2/3$  from the above recursive formula.   Our initial values are
\[
E_1=1,~~ Q_1=1, ~~ K_1=-2.
\]

For $n=1$, $E_2=2(1+a_1)$, or equivalently, $a_1=\frac{E_2}{2}-1$.  Using \eqref{Kn+1}, we have
\[
K_2=(1+4a_1)(-2)-12a_1^2-8a_1-2=-4(1+a_1)(1+3a_1)=4E_2-3E_2^2.
\]

Applying \eqref{Kn+1} with $n=2$ yileds
\begin{eqnarray*}
K_3
&=&(1+2a_2)K_2-3a_2^2E_2^2-4a_2E_2-2\\
&=&(1+2a_2)(4E_2-3E_2^2)-3a_2^2E_2^2-4a_2E_2-2\\
&=&-\frac23-\frac13\Big(3(1+a_2)E_2-2\Big)^2.
\end{eqnarray*}
Since $E_3=(1+a_2)E_2+1$ from \eqref{En+1}, we have

\[
K_3=
-\frac23-\frac13\left(3E_3-5\right)^2=-3E_3^2+10E_3-9.
\]


Using \eqref{Kn+1} with $n=3$, we obtain
\[
K_4
=
\left(1+\frac43a_3\right)K_3
-\frac43a_3^2E_3^2
-\frac83a_3E_3-2=-\frac13G(E_3,a_3),
\]
where function $G(x,y)$ is defined as
\[
G(x,y)
=
4x^2y^2
+(12x^2-32x+36)y
+9x^2-30x+33.
\]

Our goal is to show $K_3\le 0$.  
Since $E_3=1+(1+a_2)M_2\ge 1$ and $-1\le  a_3\le 1$, 
we need to verify that
\[
G(x,y)\ge 0 ~~\text{ for } x\ge 1,~ -1\le y\le 1.
\]

For each $x$, $G(x,y)$ is a quadratic function of $y$, and the global minimum is achieved at
\[
\bar{y}_x=
-\frac{12x^2-32x+36}{8x^2}
=
-\frac{3x^2-8x+9}{2x^2}.
\]
If $\bar{y}_x\notin [-1,1]$, the minimum over $[-1,1]$ is attained at an endpoint. In fact, at the endpoints,
\[
G(x,-1)
=
(x-1)(x+3)\ge 0, ~~
G(x,1)
=
25x^2-62x+69>0.
\]
Note that $\bar{y}_x<0$ for all $x\ge 1$.
The range of $x$ for $\bar{y}_x\in [-1,1]$ is $[4-\sqrt{7}, 4+\sqrt{7}]$.
Within this range,
\[
G(x,\bar{y}_x)
=
\frac{18x^3-85x^2+144x-81}{x^2}\ge 0
\]
can be easily verified since the numerator for the expression of $G(x,\bar{y}_x)$ is increasing. 

We conclude from the above discussion that 
$K_4\le 0$.

Now we employ \eqref{Kn+1} with $n=4$ and get
\begin{eqnarray*}
K_5
&=&
(1+a_4)K_4
-\frac34a_4^2E_4^2
-2a_4E_4-2\\
&=&(1+a_4)K_4
-\frac34(a_4E_4+\frac43)^2-\frac{2}{3}\\
&\le& -\frac23
\end{eqnarray*}
since $(1+a_4)K_4\le 0$.

\noindent\textbf{Step 2}.  Define 
\[
\kappa_n=\E(M_n^4)-3\big(\E(M_n^2)\big)^2.
\]
We will show 
\begin{equation}\label{kappa}
\kappa_\infty:=\E(L^4)-3\big(\E(L^2)\big)^2<0
\end{equation}
under condition $\lim_{n\to\infty}V_n<\infty$,
where $L$ is the limit random variable in Theorem~\ref{thm4}. 

Recall $M_n=A_nS_n$ for all $n\ge 4$.  We have $\kappa_n=A_n^4K_n$ and obtain from \eqref{Kn+1} that for $n\ge 5$

\begin{equation}\label{kappa+1}
\kappa_{n+1}
=\frac{1+\frac{4a_n}{n}}{(1+\frac{a_n}{n})^4}\kappa_n-A_{n+1}\left(\frac{12a_n^2}{n^2}E_n^2
+\frac{8a_n}{n}E_n+2\right)<\frac{1+\frac{4a_n}{n}}{(1+\frac{a_n}{n})^4}\kappa_n
\end{equation}
 since $A_{n+1}>0$ and
\[
\frac{12a_n^2}{n^2}E_n^2
+\frac{8a_n}{n}E_n+2=\frac{1}{3}\left(\frac{6a_nE_n}{n}+2\right)^2+\frac23>0.
\]
Employing \eqref{kappa+1} repeatedly ,  we have for all $n>5$
\[
\kappa_{n+1}<\kappa_5\prod^n_{k=5}\frac{1+\frac{4a_k}{k}}{(1+\frac{a_k}{k})^4}.
\]
Since $M_n$ converges in $\mathcal{L}^4$ to $L$, we have 
\begin{equation}\label{momoment-i}
\lim_{n\to\infty}\E(M_n^i)=\E(L^i),~~~ i=1,\ldots, 4,
\end{equation}
and thus $\lim_{n\to\infty}\kappa_{n+1}= \kappa_{\infty}$. Therefore, by letting $n$ go to infinity in the above equation, we conclude that
\begin{equation}\label{kappa-inf}
\kappa_{\infty}\le \kappa_5\prod^\infty_{k=5}\frac{1+\frac{4a_k}{k}}{(1+\frac{a_k}{k})^4}.
\end{equation}
Note that for $k\ge 5$, $\frac{1+\frac{4a_k}{k}}{(1+\frac{a_k}{k})^4}>0$, and
\[
|1-\frac{1+\frac{4a_k}{k}}{1+\frac{a_k}{k})^4}|=
\frac{1}{(1+\frac{a_k}{k})^4}|(1+\frac{a_k}{k})^4-(1+\frac{4a_k}{k})|\le 11(\frac{5}{4})^4k^{-2},
\]
which implies
\[
\sum^\infty_{k=5}|1-\frac{1+\frac{4a_k}{k}}{1+\frac{a_k}{k})^4}|<\infty.
\]
 Consequently, the limit $\prod^\infty_{k=5}\frac{1+\frac{4a_k}{k}}{(1+\frac{a_k}{k})^4}\in (0,\infty)$. \eqref{kappa} follows from \eqref{kappa-inf} since $\kappa_5=K_5/A_5=-2A_5/3<0$.

\noindent\textbf{Step 3}. Finally, by our assumption,  $\E(M_n)=p_1(2q-1)=0$ and thus $\E(L)=0$. Therefore,  $\kappa_\infty$ defined in \eqref{kappa} is the fourth cumulant of $L$. This implies that $L$ is not normal since the fourth cumulant of any normal random variable must be equal to zero. This completes the proof.   
\end{proof}

\begin{proof}[Proof of Theorem~\ref{thm5}(\textbf{b})]

Set  $\mu=\E(X_1)=2q-1$. Then $\E(M_1)=p_1\mu$. Since $\{M_n,~\F_n; n\ge 1\}$ is a martingale sequence,  $\E(M_n)=\E(M_1)=p_1\mu$, and  
$\E(L)=\mu$ from \eqref{momoment-i}.

Now we define $C_n=\E\big(M_n-p_1\mu\big)^3$, the third central moment of $M_n$. Also define $C_{\infty}=\E\big(L-p_1\mu\big)^2$.  Again, from \eqref{momoment-i} we have 
\[
\lim_{n\to\infty}C_n=C_{\infty}\in (-\infty,\infty).
\]
Our goal is to prove $C_\infty\ne 0$ under assumption that $q\ne \frac12$ and $p_n\ge \frac12$ for all $n\ge 1$.  This is sufficient to 
conclude that $L$ is not normal since the third central moment of a normal random variable must be zero.

We will establish a recursive formula for $C_n$ first.

Recall that 
\[
\delta_{n+1}=A_{n+1}\left(X_{n+1}-a_{n}\frac{S_n}{n}\right)
\]
as in \eqref{deltan}. We have from  \eqref{Conditional-2} that
\[
Q_n:=\E(\delta_{n+1}^2\mid\F_n)
=A_{n+1}^2\left(1-\frac{a_n^2S_n^2}{n^2}\right).
\]
By using \eqref{conditional-mean} and the fact that $X_{n+1}^2=1$, we can show 
\[
D_n:=\E(\delta_{n+1}^3\mid\F_n)
=-2A_{n+1}^3\frac{a_nS_n}{n}
\left(1-\frac{a_n^2S_n^2}{n^2}\right).
\]

Set $Y_n=M_n-p_1\mu$. 
Since $Y_{n+1}=Y_n+\delta_{n+1}$, we have
\[
Y_{n+1}^3=Y_{n}^3+3Y_{n}^2\delta_{n+1}+3Y_n\delta_n^2+\delta_n^3.
\]
Taking conditional expectations on $\F_n$, using $\E(\delta_{n+1}\mid\F_n)=0$, and then taking expectations we get
\[
C_{n+1}=C_n+3\E(Y_nQ_n)+\E(D_n).
\]
By substituting $M_n=A_nS_n$, using $A_{n+1}=nA_n/(n+a_n)$ from \eqref{An+1}, and simplifying, we have the recursive formula of $C_n$ for $n\ge 2$
\begin{equation}\label{rec}
C_{n+1}=c_nC_n+e_n,
\end{equation}
where
\[
c_n=\frac{n^2(n+3a_n)}{(n+a_n)^3}, ~~e_n=
\frac{2p_1\mu}{(n+a_n)^3}
\Big(
-3a_n^2nv_n-a_nn^2A_n^2+a_n^3p_1^2\mu^2\Big)~~~\text{ with } v_n:=\Var(M_n).
\]

Since $\mu=2q-1\in [-1,1]$,  $q\ne \frac12$ implies $\mu\ne 0$. 
Our initial is $C_2=p_1\mu(1-3p_1+2p_1^2\mu^2)$, which can be easily obtained by using 
\[
\E\big(M_2^j\big)=\E\Big(\E\big((\frac{S_2}2)^j|\F_1\big)\Big)=p_1\E\big(X_1^j\big)=\begin{cases}
p_1\mu,&j=1,~3,\\
p_1,   &j=2,
\end{cases}
\]
expanding and simplifying 
$C_2=\E\big(M_2-p_1\mu\big)^3$.


By repeatedly using \eqref{rec}, we have
\[
C_n
=
C_2\prod_{j=2}^{n-1}c_j
+
\sum_{k=2}^{n-1}
e_k\prod_{j=k+1}^{n-1}c_j, ~~n\ge 3,
\]
which is equivalent to
\[
\bar{C}_n
=
\bar{C}_2\prod_{j=2}^{n-1}c_j
+
\sum_{k=2}^{n-1}
\bar{e}_k\prod_{j=k+1}^{n-1}c_j, ~~n\ge 3,
\]
where
\[
\bar{C}_n=\frac{C_n}{p_1\mu}, ~~
\bar{C}_\infty=\frac{C_\infty}{p_1\mu}, ~~\bar{e}_n=\frac{e_n}{p_1\mu}=
\frac{2a_n}{(n+a_n)^3}
\Big(
-3a_nnv_n-n^2A_n^2+a_n^2p_1^2\mu^2\Big).
\]

Define
\[
P_k=\prod_{j=k}^{\infty}c_j, ~~k\ge 2.
\]
Under our assumption $a_n\ge 0$,  $c_n>0$ and 
\[
c_n=1-\frac{3a_n^2}{n^2}+O(n^{-3}),
\]
which implies $\sum^\infty_{n=2}|1-c_n|<\infty$,
all infinite products $P_k\in (0,\infty)$ are well defined.  
By letting
$n$ tend to infinity, we have
\begin{equation}\label{Cinf}
\bar{C}_\infty
=
\bar{C}_2P_2
+\sum_{n=2}^\infty \bar{e}_nP_{n+1}.
\end{equation}
The right-hand side of \eqref{Cinf} is also well defined since $P_n>0$ and $e_n\le 0$ for all $n\ge 2$.  In fact,  since $p_n\ge 1/2$, we have $a_n=2p_n-1\in [0,1]$, $nA_n\ge 2A_2=1$ due to the monotonicity of $nA_n$ from Lemma~\ref{lem1}, and $|\mu|\le 1$, and consequently,  $n^2A_n^2\ge a_n^2p_1\mu^2$, and 
\[
\bar{e}_n=\frac{e_n}{p_1\mu}=
\frac{2a_n}{(n+a_n)^3}
\Big(
-3a_nnv_n-n^2A_n^2+a_n^2p_1^2\mu^2\Big)\le 0
\]
for all $n\ge 2$.

We have $\bar{C}_2=1-3p_1+2p_1^2\mu^2$ as we have shown that $C_2=p_1\mu(1-3p_1+2p_1^2\mu^2)$.

To complete the proof, it suffices to show $\bar{C}_\infty<0$, which implies $C_{\infty}\ne 0$.
We consider two cases: Case \textbf{1}: $|\mu|<1$, and  Case \textbf{2}: $|\mu|=1$.

    Define $f(x)=1-3x+2x^2=(x-1)(2x-1)$. clearly, $f(x)\le 0$ for $x\in [1/2, 1]$. 

    For case 1,  $\bar{C}_2<f(p_1)\le 0$ for $1/2\le p_1\le 1$. Since all $\bar{e}_n\le 0$, it follows from \eqref{Cinf} that $\bar{C}_\infty\le \bar{C}_2P_2<0$.

   For Case 2,  $\bar{C}_2=f(p_1)\le 0$ for $1/2\le p_1 \le 1$. Therefore, all summands on the right-hand side of \eqref{Cinf} are non-positive. But we will show at least one of these summands is negative.   Since $|\mu|=1$ implies $q=0$ or $1$, that is, $X_1$ is degenerate.  Since we assume $L$ is non-degenerate, we have $\Var(L)>0$. Therefore, we have
   $v_n=\Var(M_n)\to \Var(L)>0$, which implies $v_n>0$ for all large $n$.  It is true that $a_n>0$ for infinitely many $n$, otherwise there exists an $n_0>2$ such that $a_n=0$ for all $n\ge n_0$. In this case, $A_n=A_{n_0}>0$ for all $n>n_0$, and $\lim_{n\to\infty}V_n=\sum^\infty_{n=2}A_n^2=\infty$,  contradictory to the assumption in Theorem~\ref{thm4}.  Now we conclude that for all large $n$ with $a_n>0$, 
   \[
\bar{e}_n=
\frac{2a_n}{(n+a_n)^3}
\Big(
-3a_nnv_n-n^2A_n^2+a_n^2p_1^2\mu^2\Big)
\le \frac{-6a_n^2nv_n}{(n+a_n)^3}<0, 
\]
yielding that $\bar{C}_{\infty}<0$.  This completes the proof of Theorem~\ref{thm5}.
\end{proof}


\end{document}